\documentclass[twoside,11pt]{amsart}

\usepackage{amsmath, amsfonts,latexsym, amssymb, times, enumerate}
\usepackage{MnSymbol}
\usepackage{tikz}
\usepackage{hyperref}
\usepackage[all]{xy}

\input amssym.def
\input amssym
\input xypic
\input xy
\xyoption{all}
\def\demo{\noindent{\bf Proof. }}
\def\sqr#1#2{{\vcenter{\hrule height.#2pt
        \hbox{\vrule width.#2pt height#1pt \kern#1pt
                \vrule width.#2pt}
        \hrule height.#2pt}}}
\def\square{\mathchoice\sqr64\sqr64\sqr{4}3\sqr{3}3}
\def\QED{\hfill$\square$}

\newtheorem{Theorem}{Theorem}[section]

\newtheorem{Lemma}[Theorem]{Lemma}
\newtheorem{Corollary}[Theorem]{Corollary}
\newtheorem{Proposition}[Theorem]{Proposition}

\newtheorem{Conjecture}[Theorem]{Conjecture}

\newtheorem{Notation and Discussion}[Theorem]{Notation and Discussion}

\newtheorem{Set-up}[Theorem]{Set-up}
\newtheorem{Assumptions and Discussion}[Theorem]{Assumptions and Discussion}

\newtheorem{Remark}[Theorem]{Remark}
\newtheorem{Example}[Theorem]{Example}

\newtheorem{Question}[Theorem]{Question}

\newcommand{\ul}[1]{\underline{#1}}

\newcommand{\PN}{\mathbb{P}_{k}^N}

\def\Ass{{\rm Ass}}

\begin{document}

\baselineskip=16pt

\title[Chudnovsky's conjecture for very general points in $\mathbb{P}_k^N$]
{\Large\bf Chudnovsky's conjecture for very general points in $\mathbb{P}_k^N$}

\author[L. Fouli, P. Mantero  and Y. Xie]
{Louiza Fouli, Paolo Mantero   \and Yu Xie}

\address{Department of Mathematical Sciences,  New Mexico State University, Las Cruces, New Mexico 88003} \email{lfouli@nmsu.edu}

 \address{Department of Mathematical Sciences,  University of Arkansas, Fayetteville, Arkansas  72701} \email{pmantero@uark.edu}

\address{Department of Mathematics, Widener University,
Chester, Pennsylvania 19013} \email{yxie@widener.edu}

\address{} \email{}

\


\thanks{The first author was partially supported by a grant from the Simons Foundation, grant \#244930. }

\vspace{-0.1in}

\begin{abstract}
We prove a long-standing conjecture of Chudnovsky for very general and generic points in $\mathbb{P}_k^N$, where $k$ is an algebraically closed field of characteristic zero, and for any finite set of points lying on a quadric, without any assumptions on $k$. We also prove that for any homogeneous ideal $I$ in the homogeneous coordinate ring $R=k[x_0, \ldots, x_N]$, Chudnovsky's conjecture holds for large enough symbolic powers of $I$.
\end{abstract}

\maketitle

\vspace{-0.2in}

\section{Introduction}
This manuscript deals with the following general interpolation question:

\begin{Question}\label{alpham}
 Given a finite set of $n$ distinct points $X=\{\mathfrak{p_1},\ldots,\mathfrak{p_n}\}$ in $\mathbb{P}_k^N$, where $k$ is a field, what is the minimum degree, $\alpha_m(X)$, of a hypersurface $f\neq 0$ passing through each $\mathfrak{p_i}$ with multiplicity at least $m$?
 \end{Question}

Question~\ref{alpham}  has been considered in various forms for a long time. We mention a few conjectures and motivations. For instance, this question plays a crucial role in the proof of Nagata's counterexamples to Hilbert's fourteenth problem \cite{Na}. In the same paper Nagata conjectured that $\alpha_m(X)>m\sqrt{n}$ for sets of $n$ general points in $\mathbb P_{\mathbb C}^2$ \cite{Na}, and a vast number of papers in the last few decades are related to his conjecture.
Another reason for the interest sparked by the above question comes from the context of complex analysis: an answer to Question~\ref{alpham} would provide information about the Schwarz exponent, which is very important in the investigation of the arithmetic nature of values of Abelian functions of several variables \cite{Ch}.

However, besides a few very special classes of points (e.g., if these $n$ points lie on a single hyperplane or one has $n=\binom{\beta+N-1}{N}$ points forming a star configuration and $m$ is a multiple of $N$ \cite{De},\cite{BH}), at the moment a satisfactory answer to this elusive question appears out of reach.
Therefore, there has been interest in finding effective lower bounds for $\alpha_m(X)$. In fact, lower bounds for $\alpha_m(X)$ yield upper bounds for the Schwarz exponent.
Using complex analytic techniques, Waldschmidt \cite{W} and Skoda \cite{Sk} in 1977  proved that for all $m\geq 1$
$$
\frac{\alpha_m(X)}{m}\geq \frac{\alpha(X)}{N},
$$
where $\alpha(X)=\alpha_1(X)$ is the minimum degree of a hypersurface passing through every point of $X$ and $k=\mathbb C$.
In 1981, Chudnovsky ~\cite{Ch} improved the inequality in the 2-dimensional projective space. He showed that if $X$ is a set of $n$ points in $\mathbb{P}_{\mathbb{C}}^2$,\, then for all $m\geq 1$
 $$
 \displaystyle\frac{\alpha_m(X)}{m}\geq \frac{\alpha(X)+1}{2}.
 $$
 He then raised the following conjecture for higher dimensional projective spaces:
\begin{Conjecture}[Chudnovsky \cite{Ch}]\label{CC1}
\,\, If $X$ is a finite set of points in $\mathbb{P}_{\mathbb{C}}^N$, then for all $m\geq 1$
$$
\displaystyle\frac{\alpha_m(X)}{m}
\geq \frac{\alpha(X)+N-1}{N}.
$$
\end{Conjecture}
The first improvement towards  Chudnovsky's Conjecture~\ref{CC} was achieved in \cite{EV} by
Esnault and Viehweg, who employed complex projective geometry techniques to show
$\displaystyle\frac{\alpha_m(X)}{m}\geq \frac{\alpha(X)+1}{N}$ for points in $\mathbb{P}_{\mathbb{C}}^N$.
In fact, this inequality follows by a stronger statement, refining previous inequalities from Bombieri, Waldschmidt and Skoda, see \cite{EV}.

From the algebraic point of view,  Chudnovsky's Conjecture~\ref{CC} can be interpreted in terms of symbolic powers via a celebrated theorem of Nagata and Zariski. Let $R=k[x_0,\ldots,x_N]$ be the homogeneous coordinate ring of $\mathbb{P}_{k}^N$ and $I$ a homogeneous ideal in $R$.
We recall that the $m$-th {\em symbolic power} of  $I$ is defined as the ideal $I^{(m)}=\bigcap_pI^mR_p\cap R$, where $p$ runs over all associated prime ideals of $R/I$, and the {\em initial degree} of $I$, $\alpha(I)$, is the least degree of a polynomial in $I$. Nagata and Zariski showed that if $k$ is algebraically closed and $X$ is a finite set of points in $\mathbb{P}_k^N$, then $\alpha_m(X)=\alpha(I_X^{(m)})$, where $I_{X}$ is the ideal consisting of  all polynomials in $R$ that vanish on $X$. Thus  in this setting Chudnovsky's Conjecture~\ref{CC} is equivalent to
$$
\frac{\alpha(I_X^{(m)})}{m}\geq \frac{\alpha(I_X)+N-1}{N} \mbox{\ \ \ for all } m \geq 1.
$$

The limit $\displaystyle\lim\limits_{m\rightarrow \infty}\frac{\alpha(I_X^{(m)})}{m}=\gamma(I_X)$, called the {\em Waldschmidt constant} of $I_X$, exists and is an ``inf'' \cite{BH}. Thus another equivalent formulation of Chudnovsky's Conjecture~\ref{CC} is $$\gamma(I_X)\geq \frac{\alpha(I_X)+N-1}{N}.$$
We remark here that there is a tight connection between the Waldschmidt constant (especially for general points) and an instance of the  multipoint Seshadri constant \cite[Section 8]{Sesh}.

We now state a generalized  version of Chudnovsky's Conjecture~\ref{CC1}. When $k$ is an algebraically closed field then the following conjecture is equivalent to Chudnovsky's Conjecture~\ref{CC1}.

\begin{Conjecture}\label{CC}
If $X$ is a finite set of points in $\mathbb{P}_{{k}}^N$, where $k$ is any field, then for all $m\geq 1$
$$
\displaystyle\frac{\alpha(I_X^{(m)})}{m}
\geq \frac{\alpha(I_X)+N-1}{N}.
$$
\end{Conjecture}

In 2001, Ein, Lazarsfeld, and Smith proved a containment between ordinary powers and symbolic powers of homogeneous ideals in polynomial rings over the field of complex numbers. More precisely, for any homogeneous ideal $I$ in $R=\mathbb{C}[x_0, \ldots, x_N]$, they proved that $I^{(Nm)} \subseteq I^m$ \cite{ELS}. Their result was soon generalized over any field by Hochster and Huneke using characteristic $p$ techniques~\cite{HoHu}.
Using this result, Harbourne and Huneke  observed that   the Waldschmidt-Skoda inequality $\displaystyle\frac{\alpha(I^{(m)})}{m}\geq \frac{\alpha(I)}{N}$ actually holds for every homogeneous ideal $I$ in $R$ \cite{HaHu}. In the same article, Harbourne and Huneke posed the following conjecture:
\begin{Conjecture}[Harbourne-Huneke \cite{HaHu}]\label{HH}
If $X$ is a finite set of points in $\mathbb{P}_k^N$, then for all $m\geq1$
$$
I_X^{(Nm)}\subseteq M^{m(N-1)}I_X^m,
$$
where $M=(x_0, \ldots, x_N)$ is the  homogeneous maximal ideal of $R=k[x_0, \ldots, x_N]$.
\end{Conjecture}

Conjecture~\ref{HH} strives to provide a  structural reason behind  Chudnovsky's Conjecture~\ref{CC}: if it holds, then it would imply  Chudnovsky's Conjecture~\ref{CC} in a similar way as to how the Ein-Lazarsfeld-Smith and Hochster-Huneke containment implies the Waldschmidt-Skoda inequality \cite{HaHu}.
These results have since raised new interest in  Chudnovsky's Conjecture~\ref{CC}.

Harbourne and Huneke proved their conjecture for general points in $\mathbb{P}_k^2$ and when the points form a star configuration in $\PN$.
In 2011, Dumnicki proved  the Harbourne-Huneke Conjecture~\ref{HH} for general points in $\mathbb{P}_k^3$ and at most $N+1$ points in general position in $\mathbb{P}_k^N$ for $N\geq 2$ \cite{D}.  In summary,  Chudnovsky's Conjecture~\ref{CC}  is  known in the following cases:
\begin{itemize}
\item any finite set of points in $\mathbb{P}_k^2$ \cite{Ch}, \cite{HaHu};
\item any finite set of general points in $\mathbb{P}_k^3$, where $k$ is a field of characteristic 0 \cite{D};
\item any set of at most $N+1$ points in general position in $\mathbb{P}_k^N$ \cite{D};
\item any set of a binomial number of points in $\mathbb{P}_{k}^N$ forming a star configuration \cite{De}, \cite{BH}.
\end{itemize}

In the present paper, we prove that  Chudnovsky's Conjecture~\ref{CC} holds for
\begin{itemize}
\item any finite set of very general points in $\mathbb{P}_k^N$, where $k$ is an algebraically closed field of characteristic 0  (Theorem~\ref{main2});
\item any finite set of generic points in $\mathbb{P}_{k (\underline{z})}^N$, where $k$ is an algebraically closed field of characteristic 0 (Theorem~\ref{generic2});
\item any finite set of points in $\mathbb{P}_k^N$ lying on a quadric, without any assumptions on $k$ (Proposition~\ref{2N2}).
\end{itemize}
 As a corollary, we  obtain that the Harbourne-Huneke Conjecture~\ref{HH} holds for sets of a binomial number of very general points in $\mathbb{P}_{k}^N$ (Corollary~\ref{HH2}). This result also yields a new lower bound for the multipoint Seshadri constant of very general points in $\mathbb{P}_k^N$ (Corollary~\ref{Seshadri2}).

In the final section of the paper, we prove that for any homogeneous ideal $I$ in the homogeneous coordinate ring $R=k[x_0, \ldots, x_N]$,  Chudnovsky's Conjecture~\ref{CC} holds for sufficiently large symbolic powers $I^{(t)}$, $t\gg 0$ (Theorem~\ref{HomogeneousIdeal3}). In the case of ideals of  finite  sets of points in $\mathbb{P}_{\mathbb{C}}^N$, we prove a uniform bound, namely that if $t\geq N-1$, then $I^{(t)}$ satisfies  Chudnovsky's Conjecture~\ref{CC} (Proposition~\ref{HomogeneousIdeal11}).

Very recently, Dumnicki and Tutaj-Gasi\'{n}ska proved the Harbourne-Huneke Conjecture~\ref{HH} for at least $2^N$ number of very general points in $\mathbb{P}_k^N$. As a corollary, they obtain Chudnovsky's Conjecture~\ref{CC} for at least $2^N$ number of very general points in $\mathbb{P}_k^N$ \cite{D1}. These results are obtained independently from ours and with different methods.

\section{Generic and very general points in $\mathbb{P}_k^N$}

 We begin by discussing our general setting.

\begin{Set-up}\label{setup2}
Let $R=k[x_0,\ldots, x_N]$ be the homogeneous coordinate ring of $\mathbb{P}_k^N$, where $k$ is an algebraically closed field. Let  $n$ be a positive integer and let $S=k(\underline{z}) [\underline{x}]$, where $k\subseteq k(\underline{z})$ is a purely transcendental  extension of fields obtained by adjoining  $n(N+1)$   variables $\underline{z}=(z_{ij})$, $1\leq i\leq n, \,0\leq j\leq N$.
A set of $n$ generic points $P_1,\ldots,P_n$ consists of points
$P_i=[z_{i0}:z_{i1}:\ldots:z_{iN}]\in \mathbb{P}^N_{k(\underline{z})}$. We denote the defining ideal  of $n$ generic points as
$$
H=\bigcap \limits_{i=1}^nI(P_i),
$$
where $I(P_i)$ is the ideal defining the point $P_i$.

For any nonzero vector $\underline{\lambda}=(\lambda_{ij}) \in \mathbb{A}_k^{n(N+1)}$, where $1\leq i\leq n, 0\leq j\leq N$,  we define the set of points $\{p_1,\ldots,p_n\}\subseteq \mathbb P_k^N$ as the points  $p_i=P_i(\underline{\lambda})=[\lambda_{i0}: \lambda_{i1}: \ldots: \lambda_{iN}]\in \mathbb P_k^N$.
For $1 \leq i \leq n$, let $I(p_i)$ be the ideal of $R$ defining the point $p_i$ and define
$$
H ({\ul\lambda})=\bigcap \limits_{i=1}^nI(p_i).
$$
\end{Set-up}

For any  ideal $J$ in $S$,  recall that Krull \cite{Kr1, Kr2} defined the specialization $J_{\underline{\lambda}}$ with respect to the substitution $\underline{z}\rightarrow \underline{\lambda}$ as follows:
$$
J_{\underline{\lambda}}=\{f(\underline{\lambda}, \underline{x})\,|\, f(\underline{z}, \underline{x})\in J\cap k[\underline{z}, \underline{x}]\}.
$$
In general, one has that $H_{\underline{\lambda}}\subseteq H({\underline{\lambda}})$, where $H$ and $H(\underline{\lambda})$ are defined as in Set-up~\ref{setup2} and
$H_{\underline{\lambda}}$ is the specialization with respect to the substitution $\underline{z}\rightarrow \underline{\lambda}$ defined by Krull.
Notice that equality holds if $\underline{\lambda}$ is in a dense Zariski-open  subset  of $\mathbb{A}_k^{n(N+1)}$  \cite{NT}.

Recall the collection of all sets  consisting of $n$  points (not necessarily distinct) in $\mathbb{P}_k^N$ is parameterized by  $G(1,n,N+1)$,  the {\em Chow variety} of algebraic $0$-cycles of degree $n$ in $\mathbb P_k^N$. It is well-known that
 $G(1,n,N+1)$ is isomorphic to  the symmetric product ${\rm Sym}^n(\mathbb P_k^N)$, see for instance \cite{GKZ}.   One says that a property $\mathcal P$ holds for $n$ {\it  general points} in $\mathbb{P}_{k}^N$ if there is a dense Zariski-open subset $W$ of $G(1,n,N+1)$ such that $\mathcal P$ holds for every set $X=\{p_1,\ldots,p_n\}$ of $n$ points with $p_1+\ldots+p_n\in W$. Similarly, one says that a property $\mathcal P$ holds for $n$ {\it very general points} in $\mathbb{P}_{k}^N$ if $\mathcal P$ holds for every set of $n$ points in a nonempty subset $W$  of $G(1,n,N+1)$ of the form $W=\bigcap \limits_{i=1}^{\infty} U_i$, where the $U_i$ are dense Zariski-open sets (when $k$ is uncountable, then $W$ is actually a dense subset). We conclude this part by recalling the following well-known fact.

 \begin{Remark}\label{Xn2}{\rm
Let $n$ be a positive integer. The collection of all sets  consisting of $n$ distinct points in $\mathbb{P}_k^N$ is parameterized by a dense Zariski-open subset $W(n)$ of $G(1,n,N+1)$.}
\end{Remark}

Unless specified, for the rest of this paper by a ``set of points'' we mean ``a set of simple points'', i.e. points whose defining ideal is radical.

Instead of working directly with the Chow variety, we will work over $\mathbb A_k^{n(N+1)}$ (in order to specialize from the generic situation). We first need to prove that if a property holds on a dense Zariski-open subset of $\mathbb A_k^{n(N+1)}$, then it also holds on a dense Zariski-open subset of the Chow variety. This is precisely the content of our first lemma.

\begin{Lemma}\label{open}
Assume Set-up~\ref{setup2} and let $U\subseteq \mathbb A_k^{n(N+1)}$ be a  dense Zariski-open subset such that a property $\mathcal P$ holds for $H(\underline{\lambda})$ whenever $\ul{\lambda}\in U$. Then property $\mathcal P$ holds for $n$ general points in $\mathbb P_k^N$. Moreover, if a property $\mathcal P$ holds for $H(\underline{\lambda})$ whenever $\ul{\lambda}\in U$, where $U=\bigcap \limits_{i=1}^{\infty} U_i \subseteq \mathbb{A}_k^{n(N+1)}$ is nonempty and each $U_i $ is a dense Zariski-open set, then $\mathcal{P}$ holds for $n$ very general points in $\mathbb{P}_k^N$.
\end{Lemma}

\demo For every $i=1,\ldots,n$, let $
  \xymatrix{
\pi_i:\mathbb A_k^{n(N+1)}   \ar@{-->}[r] & \mathbb P_k^N
  }$
be the rational map defined by projection as follows:
$$
\pi_i\left(\underline{\lambda}\right)=[\lambda_{i0}:\lambda_{i1}:\ldots:\lambda_{iN}],
$$
where $\underline{\lambda}=(\lambda_{ij})\in \mathbb A_k^{n(N+1)}$.
It is clear that $\pi_i$ is defined on the complement of the Zariski-closed proper subset $C_i=\{\underline{\lambda}\in \mathbb A_k^{n(N+1)}\,|\,\lambda_{i0}=\ldots=\lambda_{iN}=0\}$.

Taking products of these rational maps, we obtain the rational map
$$\pi=(\pi_1\times\pi_2\times\cdots\times\pi_n): \xymatrix{\mathbb A_k^{n(N+1)}   \ar@{-->}[r] &  \mathbb P_k^N \times_k \mathbb P_k^N \times_k \cdots \times_k \mathbb P_k^N.
  }$$
The map $\pi$ is  defined on the complement of the closed proper subset $C=\bigcup \limits_{i=1}^n C_i$, where $C_i$ is as above. Note that $U\setminus C$ is still
open in $\mathbb A_k^{n(N+1)}$,  and since $\pi$ is surjective and thus dominant, then $\pi(U\setminus  C)$ contains
a non-empty Zariski-open subset $W^{\prime}\subseteq \mathbb P_k^N \times_k \mathbb P_k^N \times_k \cdots \times_k \mathbb P_k^N$
(see for instance \cite[II. Ex. 3.19~(b)]{Ha}).

Now, since the symmetric group $S_n$ on $n$ elements is finite,  the image $W$ of $W^{\prime}$ in $(\mathbb P_k^N \times_k \mathbb P_k^N \times_k \cdots \times_k \mathbb P_k^N)/S_n\cong {\rm Sym}^n(\mathbb P_k^N)\cong G(1,n,N+1)$ contains a non-empty Zariski-open subset of $G(1,n,N+1)$.
\QED

\bigskip

Let $H$ be as in  Set-up~\ref{setup2}. We now prove that the initial degree of any symbolic power of $H$ is no smaller than the initial degree of any ideal of a set with the same number of points. Equivalently, if $I$ is the defining ideal of a set of $n$ points in $\PN$, then $\alpha(H^{(m)})\geq \alpha(I^{(m)})$ for all $m\geq 1$.

\begin{Theorem}\label{Vt2}
Let  $m\geq 1$. Assume Set-up~\ref{setup2} and that $k$ has characteristic $0$. Then $\alpha(H^{(m)})\geq \alpha(I_X^{(m)})$, for every set $X$ of $n$ distinct points in $\PN$. Moreover, for every $m\geq 1$, there is a dense Zariski-open subset $U_m\subseteq\mathbb{A}_k^{n(N+1)}$ for which equality holds.
\end{Theorem}

\demo
  Let $t \geq 0$. We define $
V_t=\{\underline{\lambda}=(\lambda_{ij})\in \mathbb A_k^{n(N+1)} \mid \alpha(H({\underline{\lambda}})^{(m)}) \leq t\}.
$
We first prove that  $V_t$ is a closed subset of $\mathbb A_k^{n(N+1)}$.
Indeed, notice that
$$
V_t=\{\underline{\lambda}=(\lambda_{ij}) \in \mathbb A_k^{n(N+1)} \mid \mbox{ there exists } 0\neq f \in H(\underline{\lambda})^{(m)} \mbox{ of degree } t\}.
$$
Let $f \in R$ be a homogeneous polynomial with $\deg f=t$ and write $f=\sum \limits_{|\underline{\alpha}|=t} C_{\ul{\alpha}} \ul{x}^{\ul{\alpha}}$.
Since $k$ is algebraically closed of characteristic 0, the statement $f\in H({\underline{\lambda}})^{(m)}$ is
equivalent to $\partial_{\ul{\beta}}f(p_i)=0$ for all $\ul{\beta}$ with $|\ul{\beta}|\leq m-1$ and  all points $p_1,\ldots,p_n$.
Since $P_i=[\underline{z_i}]=[z_{i0}:z_{i1}:\ldots:z_{iN}]$ and $p_i=[\underline{\lambda_i}]=[\lambda_{i0}:\lambda_{i1}:\ldots:\lambda_{iN}]$, we write
$\partial_{\ul{\beta}}f(P_i)=\partial_{\ul{\beta}}f(\underline{z_i})=\sum \limits_{|\underline{\alpha}|=t} C_{\ul{\alpha}}
\partial_{\ul{\beta}}\ul{z_i}^{\ul{\alpha}}$ and $\partial_{\ul{\beta}}f(p_i)= \sum \limits_{|\underline{\alpha_i}|=t}
C_{\ul{\alpha}} \partial_{\ul{\beta}}\ul{\lambda_i}^{\ul{\alpha}}$. (For instance, $\partial_{(2,0,1)}\ul{z_i}^{(3,3,2)}=\partial_{x_0x_0x_2} x_0^3x_1^3x_2^2|_{\underline{x}=\underline{z_i}}=12x_0x_1^3x_2|_{\underline{x}=\underline{z_i}}=12z_{i0}z_{i1}^3z_{i2}$ and  $\partial_{(2,0,1)}\ul{\lambda_i}^{(3,3,2)}=12\lambda_{i0}\lambda_{i1}^3\lambda_{i2}$).

To order these equations we use, for instance, the natural deglex order in $\mathbb{N}_0^{N+1}$, i.e., $\underline{\alpha}=(\alpha_0, \ldots, \alpha_N)> \underline{\beta}=(\beta_0, \ldots, \beta_N)$ if and only if $|\underline{\alpha}|>|\underline{\beta}|$ or $|\underline{\alpha}|= |\underline{\beta}|$ and there exists $j$ such that $\alpha_i=\beta_i$ for $ i\leq j$ and $\alpha_{j+1}> \beta_{j+1}$.
Then the system of equations
$\{\partial_{\ul{\beta}}f(P_i)=0\}_{|\beta|\leq m-1,\,1\leq i\leq n}$ can be written  in the following form
$$
\mathbb{B}_{m,t}\left[C_{(t,\ldots, 0)}\; \ldots \;C_{\ul{\alpha}}\; \ldots \; C_{(0,\ldots, t)}\right]^T=\ul{0},
$$
where the rows of $\mathbb{B}_{m,t}$ are
$$\left[\partial_{\ul{\beta}}z_{i0}^t \quad\ldots\quad
\partial_{\ul{\beta}}\ul{z_i}^{\ul{\alpha}}\quad\ldots\quad \partial_{\ul{\beta}} z_{iN}^t \right],
\,\,\,\mbox{ where }1 \leq i \leq n \mbox{ and }|\ul{\beta}| \leq m-1.$$
By construction, the existence of a nonzero element $f\in H({\underline{\lambda}})^{(m)}$ of degree $t$ is equivalent to the existence of a non-trivial solution for the homogeneous system $$\left[\mathbb{B}_{m,t}\right]_{\ul{\lambda}} \left[C_{(t,\ldots, 0)}\; \ldots \;C_{\ul{\alpha}}\; \ldots \; C_{(0,\ldots, t)}\right]^T=\ul{0}.$$

 Observe that the matrix $\mathbb{B}_{m,t}$ has size $n\binom{m+N}{m-1}\times \binom{t+N}{N}$. If $n\binom{m+N}{m-1}< \binom{t+N}{N}$, then for every $\ul{\lambda}\in \mathbb A_k^{n(N+1)}$ the homogeneous system $\left[\mathbb{B}_{m,t}\right]_{\ul{\lambda}} \left[C_{(t,\ldots, 0)}\; \ldots \;C_{\ul{\alpha}}\; \ldots \; C_{(0,\ldots, t)}\right]^T=\ul{0}$ has non-trivial solutions. Therefore $V_t=\mathbb A_k^{n(N+1)}$, which is closed in $\mathbb A_k^{n(N+1)}$.

If instead, $n\binom{m+N}{m-1}\geq \binom{t+N}{N}$, then the system $\left[\mathbb{B}_{m,t}\right]_{\ul{\lambda}}
\left[C_{(t,\ldots, 0)}\; \ldots \;C_{\ul{\alpha}}\; \ldots \; C_{(0,\ldots, t)}\right]^T=\ul{0}$ has non-trivial solutions
if and only if ${\rm rank}\left[\mathbb{B}_{m,t}\right]_{\ul{\lambda}}\,<\,\binom{t+N}{N}$. This is a closed condition on
$\underline{\lambda}$ as it requires the vanishing of finitely many minors, and therefore $V_t$ is closed in $\mathbb A_k^{n(N+1)}$.

Next, let $t_0=\alpha(H^{(m)})$. The set $V_{t_0}=\{\underline{\lambda}=(\lambda_{ij}) \in \mathbb A_k^{n(N+1)} \mid \alpha(H(\underline{\lambda})^{(m)})\leq t_0\}$ contains a dense Zariski-open   subset of $\mathbb A_k^{n(N+1)}$. Indeed, let $0\neq f\in \bigcap \limits_{i=1}^n I(P_i)^m$ be such that $\deg f=t_0$. We may assume that $f(\underline{z}, \underline{x})\in k[\underline{z}][x_0, \ldots, x_N]$.  Then there exists a dense Zariski-open  subset $U_m$ of $\mathbb A_k^{n(N+1)}$ such that the polynomial $0\neq f(\underline{\lambda}, \underline{x})\in (H^{(m)})_{\ul{\lambda}}=(H(\ul{\lambda}))^{(m)}$ and $\deg f=t_0$ (since $f(\ul{z},\ul{x})\neq 0$,  there is a non-empty Zariski-open subset of specializations $\ul{z}\rightarrow \ul{\lambda}$ such that $f(\ul{\lambda},\ul{x})\neq 0$).

Finally, since  $V_{t_0}$ is a Zariski-closed subset which also  contains a dense Zariski-open  subset of $\mathbb A_k^{n(N+1)}$, then $V_{t_0}=\mathbb A_k^{n(N+1)}$, which proves the statement. The second part of the statement also follows from the above argument.
\QED
\bigskip

Following \cite[Definition~2.4]{GMR}, we say that a set $X$ of $n$ points in $\PN$ is in {\it generic position}
if it has the ``generic Hilbert function'', i.e., if $H_{R/I_X}(d)=\min\{\dim_k(R_d), n\}$ for every $d\geq 0$. Being in generic position is an open condition; indeed any set of generic (or general) points (see Set-up~\ref{setup2}) is in generic position.
We now prove a reduction argument, which will allow us to concentrate on certain binomial numbers of points.

\begin{Proposition}\label{binom2}
\begin{enumerate}[$($a$)$.]
\item Chudnovsky's Conjecture~\ref{CC} holds for any finite set of  generic points if it holds for sets of $\binom{\beta+N-1}{N}$ generic points for all $\beta\geq 1$.\\
\item Chudnovsky's Conjecture~\ref{CC} holds for any finite set of points if it holds for sets of $\binom{\beta+N-1}{N}$ points in generic position for all $\beta\geq 1$.
\end{enumerate}
\end{Proposition}

\demo (a): Let $H=\bigcap \limits_{i=1}^{n}I(P_i)$ be the defining ideal of the $n$ generic  points $P_1, \ldots, P_n$ as in Set-up~\ref{setup2}. Let $\beta\geq 1$ be the unique integer such that
$$
\binom{\beta+N-1}{N} \leq n<\binom{\beta+N}{N}.
$$
Let $t=\binom{\beta+N-1}{N}$ and let  $J=\bigcap \limits_{i=1}^{t}I(P_i)$ be the ideal defining $t$ of these generic points. Since the set $Y=\{P_1, \ldots, P_t\}$ is in generic position, in particular we have $\alpha(J)=\alpha(H)=\beta$.

Now assume  Chudnovsky's Conjecture~\ref{CC} holds for $\binom{\beta+N-1}{N}$ generic points. Then for all $m \geq 1$
$$
\displaystyle\frac{\alpha(J^{(m)})}{m}\geq \frac{\alpha(J)+N-1}{N}.
$$
Since $H^{(m)}\subseteq J^{(m)}$, one has that $\alpha(H^{(m)})\geq \alpha(J^{(m)})$ and
$$
\displaystyle\frac{\alpha(H^{(m)})}{m}\geq \displaystyle\frac{\alpha(J^{(m)})}{m}\geq \frac{\alpha(J)+N-1}{N}=\frac{\alpha(H)+N-1}{N}.
$$

(b): The proof of (b) is similar in spirit to (a). Let $X$ be any finite set of points in $\PN$, let $I_X$ be its defining ideal, and let $t=\binom{(\alpha-1)+N-1}{N}$, where $\alpha=\alpha(I_X)$. By linear independence  (e.g., by \cite[Theorem~2.5~(c)]{GMR}), there is a subset $Y\subseteq X$ of $t$ points with the property that  $H_{R/I_Y}(i)=H_{R/I_X}(i)=\dim_kR_i$ for every $i=0,\ldots,\alpha-1$; in particular $H_{R/I_Y}(\alpha-1)=t$. Since $|Y|=t$,  it follows that $H_{R/I_Y}(i)=t$ for all $i\geq\alpha$, proving that $Y$ is in generic position. Similar to (a), assume Chudnovsky's Conjecture~\ref{CC} holds for $t=\binom{(\alpha-1)+N-1}{N}$ points in generic position. Then
$$
\displaystyle\frac{\alpha(I_Y^{(m)})}{m}\geq \frac{\alpha(I_Y)+N-1}{N}
$$
for all $m\geq 1$. Since $\alpha(I_X)=\alpha=\alpha(I_Y)$ and $I_X^{(m)}\subseteq I_Y^{(m)}$, then  for all $m\geq 1$  we obtain
$$
\displaystyle\frac{\alpha(I_X^{(m)})}{m}\geq \displaystyle\frac{\alpha(I_Y^{(m)})}{m}\geq \frac{\alpha(I_Y)+N-1}{N}=\frac{\alpha(I_X)+N-1}{N}.
$$
\QED
\bigskip

Dumnicki proved Chudnovsky's Conjecture~\ref{CC} for at most $N+1$ points in general position $\mathbb{P}_k^N$ \cite{D} (this specific result does not need any assumptions on the characteristic of $k$).
The idea is that one can take them to be coordinate points so that the ideal of the points is monomial and one
can compute explicitly its symbolic powers. If one has more than $N+1$ points, the ideal of the points is almost never monomial and explicit computations of a generating set of any of its symbolic powers are nearly impossible to perform. We extend the result of Dumnicki to the case of up to $\binom{N+2}{2}-1$ points in $\mathbb{P}_k^N$.

\begin{Proposition}\label{2N2}
 Chudnovsky's Conjecture~\ref{CC} holds for any finite set of points lying on a quadric in $\PN$, where $k$ is any field. In particular, any set of $\displaystyle n\leq \binom{N+2}{2}-1$ points in $\mathbb{P}_k^N$ satisfies Chudnovsky's Conjecture~\ref{CC}.
\end{Proposition}

\demo
Let $X$ be a set of $n$ points in $\mathbb{P}_k^N$. If they all lie on a hyperplane, then Chudnovsky's Conjecture~\ref{CC} is clearly satisfied, since $\alpha(I_X^{(m)})=m$ for every $m\geq 1$. We may then assume there is no hyperplane containing all the points. Thus we can find a set $Y\subseteq X$ of $N+1$ points not on a hyperplane, i.e. in general position.
Then for all $m\geq 1$
$$\frac{\alpha(I_X^{(m)})}{m}\geq \frac{\alpha(I_Y^{(m)})}{m}\geq \frac{N+1}{N}=\frac{\alpha(I_X)+N-1}{N},$$
where the second inequality follows by \cite{D}  and the equality holds because $\alpha(I_X)=\alpha(I_Y)=2$.
\QED
\bigskip

Let us recall that a set $X$ of $\binom{N+s}{N}$ points in $\mathbb P_k^N$ form a {\em star configuration} if there are $N+s$ hyperplanes $L_1,\ldots,L_{N+s}$ meeting properly such that $X$ consists precisely of the points obtained by intersecting any $N$ of the $L_i$'s. Star configurations (in $\mathbb P_k^2$) were already considered by Nagata and they have been  deeply studied, see for instance \cite{GHM} and references within. We employ them to show that  Chudnovsky's Conjecture~\ref{CC} holds for any number of generic points.

\begin{Theorem}\label{generic2}
Let $H=\bigcap \limits_{i=1}^{n}I(P_i)$, where $P_1, \ldots, P_n$ are  $n$ generic points in $\mathbb{P}_{k(\underline{z})}^N$ defined as in Set-up~\ref{setup2}. Suppose $k$ has characteristic $0$.
Then Chudnovsky's Conjecture~\ref{CC} holds for $H$.

\end{Theorem}

\demo
By Proposition~\ref{binom2}~(a), we may  assume $n=\binom{\beta+N-1}{N}$ for some $\beta \geq 1$. Let
$\underline{\lambda}\in \mathbb{A}_k^{n(N+1)}$ be such that $H({\underline{\lambda}})$ is the defining ideal of $n$  points in $\mathbb{P}_k^N$
forming a star configuration. It is well-known that $\alpha(H(\underline{\lambda}))=\alpha(H)=\beta$  \cite[Proposition~2.9]{GHM}.
Now by Theorem~\ref{Vt2} and \cite[Corollary~3.9]{HaHu} we have that for all $m\geq 1$
$$
\frac{\alpha(H^{(m)})}{m}\geq \frac{\alpha(H(\underline{\lambda})^{(m)})}{m}\geq \frac{\alpha(H(\underline{\lambda}))+N-1}{N}=\frac{\alpha(H)+N-1}{N}.
$$
\QED
\bigskip

We are now ready to prove our main result that Chudnovsky's Conjecture~\ref{CC} holds for  any finite set of  {\it very general} points in $\mathbb P_k^N$.

\begin{Theorem}\label{main2}
Let $I$ be the defining ideal of $n$ very general points in $\mathbb{P}_{k}^N$, where $k$ is an algebraically closed field of characteristic 0. Then $I$ satisfies Chudnovsky's Conjecture~\ref{CC}.
\end{Theorem}

\demo Being in generic position is an open condition and therefore we may assume the points are in generic position. Then, as in Proposition~\ref{binom2}~(a), we may assume that $n={\beta+N-1\choose N}$ for some $\beta \geq 1$. It suffices to show that $\gamma(I)\geq \frac{\alpha(I)+N-1}{N}$.
Let $R$, $S$, $\underline z$, $\underline \lambda$, and $H$ as in Setup~\ref{setup2}.
 Consider the decreasing chain of ideals
$$
I^{(N)}\supseteq I^{(2N)}\supseteq I^{(2^2N)}\supseteq \ldots \supseteq I^{(2^sN)} \supseteq \ldots .
$$
For each $s\geq 0$, define
$$
U_s=\{\underline{\lambda}=(\lambda_{ij})\in \mathbb A_k^{n(N+1)} \mid \alpha(H({\underline{\lambda}})^{(2^sN)}) \geq 2^s\left(\alpha(H(\underline{\lambda}))+N-1\right)\}.
$$
By  the proof of Theorem~\ref{Vt2}, $U_s$ is a Zariski-open subset of $\mathbb{A}_k^{n(N+1)}$. We claim that $U_s$ is not empty. Indeed, by Theorem~\ref{Vt2}, for every $s\geq 0$, there is a dense Zariski-open subset  $W_s\subseteq\mathbb{A}_k^{n(N+1)}$ for which $\alpha(H^{(2^sN)})= \alpha(H(\underline{\lambda})^{(2^sN)})$ for every $\underline{\lambda}\in W_s$. By Theorem~\ref{generic2}, one also has that
for every $\underline{\lambda}\in W_s$,
$$
\frac{\alpha(H(\underline{\lambda})^{(2^sN)})}{2^sN}=\frac{\alpha(H^{(2^sN)})}{2^sN}\geq \frac{\alpha(H)+N-1}{N}=\frac{\alpha(H(\underline{\lambda}))+N-1}{N}.
$$
Hence $W_s \subset U_s$.

Set $ U=\bigcap \limits_{s=0}^{\infty} U_s$  and notice $U$ is non-empty because a star configuration of $n$ points lies in $U$ \cite[Lemma~2.4.2]{BH}. By construction, if  $\underline{\lambda}\in U$ we have
$$
\gamma\left(H(\underline{\lambda})\right)=\displaystyle\lim\limits_{s\rightarrow \infty}\frac{\alpha\left(H(\underline{\lambda})^{(2^sN)}\right)}{2^sN}\geq \frac{\alpha\left(H(\underline{\lambda})\right)+N-1}{N}.
$$
Finally, apply Lemma~\ref{open}.
\QED
\bigskip

As a corollary, we show that the Harbourne-Huneke Conjecture~\ref{HH} holds for sets of binomial numbers of very general points or generic points.
  \begin{Corollary}\label{HH2}
Let $I$ be the defining ideal of either $\binom{\beta+N-1}{N}$ very general  points in $\mathbb{P}_k^N$ or $\binom{\beta+N-1}{N}$  generic points in $\mathbb{P}_{k(\underline{z})}^N$ for some $\beta\geq 1$, where $k$ is an algebraically closed field of characteristic 0. Then $I$ satisfies the Harbourne-Huneke Conjecture~\ref{HH}.

\end{Corollary}

\demo
The proof  follows by Theorem~\ref{main2} and \cite[Proposition~3.3 and Remark~3.4]{HaHu}.
\QED
\bigskip

For an unmixed ideal $I$ in $R$, the {\em Waldschmidt constant} is defined by  $\gamma(I)=\lim \limits_{m\rightarrow \infty}\frac{\alpha(I^{(m)})}{m}$, see~\cite{BH} for more details.
Recall that for a finite set of points $X=\{\mathfrak{p}_1, \ldots, \mathfrak{p}_n\}$ in $\mathbb P^N_k$, the Waldschmidt constant is tightly related to the {\it (multipoint)  Seshadri constant} defined as
$$\epsilon(N,X)=\sqrt[N-1]{{\rm inf}\left\{\frac{\deg(F)}{\sum \limits_{i=1}^n{\rm mult}_{\mathfrak{p_i}}(F)}\right\}},$$
where the infimum is taken with respect to all hypersurfaces $F$ passing through at least one of the $\mathfrak{p_i}$. The study of Seshadri constants has been an active area of research for the last twenty years, see for instance the survey \cite{Sesh} and references within. Here we only note that one has $\gamma(I_X)\geq n\epsilon(N,X)^{N-1}$ and equality holds if $X$ consists of $n$ general simple points in $\mathbb P^N_{k}$. In particular, equality also holds if $X$ consists of $n$ very general simple points $\mathbb{P}_k^N$.
 Therefore, our estimate for the Waldschmidt constant also yields an estimate for the (multipoint) Seshadri constant for very general simple points of $\mathbb P^N_k$. \begin{Corollary}\label{Seshadri2}
For any set $X$ of $n$ very general points in $\mathbb P^N_k$, where $k$ is an algebraically closed field of characteristic 0,  one has
 $$\epsilon(N,X)\geq \sqrt[N-1]{\frac{\alpha(X)+N-1}{nN}}.$$
 \end{Corollary}

 \bigskip

 \section{Homogeneous ideals  in $k[x_0, \ldots, x_N]$}

 Let $R=k[x_0, \ldots, x_N]$ be the homogeneous coordinate ring of $\mathbb{P}_k^N$, where $k$ is any field, and $I$ a homogeneous ideal. For an ideal $I$ which may have embedded components, there are multiple potential definitions of symbolic powers. Following \cite{ELS} and \cite{HoHu}, we define the {\em $m$-th symbolic power} of $I$ to be
 $$I^{(m)}= \bigcap_{p\in \Ass(R/I)} (I^mR_p\cap R).$$

 Since  $I^{(Nm)} \subseteq I^m$ (see \cite{ELS} and \cite{HoHu}), one can prove that  the Waldschmidt-Skoda inequality $\displaystyle\frac{\alpha(I^{(m)})}{m}\geq \frac{\alpha(I)}{N}$  holds for every homogeneous ideal $I$ in $R$ \cite{HaHu}.
Therefore one has that $\gamma(I)\geq \frac{\alpha(I)}{N}$.
One can also prove that $\frac{\alpha(I^{(m)})}{m}\geq \gamma(I)$ for every~$m\geq 1$, see for instance \cite{BH}.

It is then natural to ask whether  Chudnovsky's Conjecture~\ref{CC} holds for any homogeneous ideal. We pose it here as an optimistic conjecture, for which we provide some evidence below:
\begin{Conjecture}\label{FMX}
Let $R=k[x_0, \ldots, x_N]$, where $k$ is any field. For any nonzero homogeneous ideal $I$ in $R$, one has
$$\displaystyle\frac{\alpha(I^{(m)})}{m}\geq \frac{\alpha(I)+N-1}{N}$$
 for every $m\geq 1$.
\end{Conjecture}

It is easy to see that if an ideal $I$ satisfies   Conjecture~\ref{FMX}  then also $I^{(t)}$ does for any $t\geq 1$. Thus in search for evidence for a positive answer to Conjecture~\ref{FMX}, one may ask whether for every homogeneous ideal $I\subseteq k[x_0,\ldots,x_N]$ there is an exponent $t_0$ such that $I^{(t)}$ satisfies Conjecture~\ref{FMX} for every $t\geq t_0$. We give a positive answer to this question in Theorem~\ref{HomogeneousIdeal3}.

We state a few lemmas before stating the main result of this section, Theorem~\ref{HomogeneousIdeal3}. The following lemma and its proof can be found in the proof of \cite[Lemma~2.3.1]{BH}.

\begin{Lemma}\label{HomogeneousIdeal2}
Let $I$ be a homogeneous ideal in $R$ and let $m\geq t$ be two positive integers. Write $m=qt + r$ for some integers $q$ and $r$ such that $0\leq r<t$. Then
$$
\frac{\alpha(I^{(m)})}{m}\leq \frac{\alpha(I^{(t)})}{t}+\frac{\alpha(I^{(r)})}{m}.
$$
In particular, if $r=0$ then we have
$\displaystyle
\frac{\alpha(I^{(tq)})}{tq}\leq \frac{\alpha(I^{(t)})}{t}.
$ \end{Lemma}

For ideals $J$ with $\Ass(R/J)={\rm Min}(J)$ it is easily verified that $\Ass(R/J^{(m)})=\Ass(R/J)$ and $(J^{(m)})^{(t)}=J^{(mt)}$ for all $m\geq 1$ and $t\geq 1$. However, when $J$ has embedded components we found examples of ideals $J$ (even monomial ideals) and exponents $m\geq 2$, $t\geq 2$ with $(J^{(m)})^{(t)} \neq J^{(mt)}$. Borrowing techniques from a very recent paper by H\`a, Nguyen, Trung and Trung \cite{HNTT} we present here an example where $(J^{(2)})^{(2)}\neq J^{(4)}$.

\begin{Example}\label{Counter}
Let $R=k[x,t,u,v]$ and $J=J_1J_2$, where
$$J_1=(x^4,x^3u,xu^3,u^4,x^2u^2v) \quad \mbox{ and } \quad J_2=(t^3, tuv, u^2v).$$
Then $(J^{(2)})^{(2)}\neq J^{(4)}$.
\end{Example}

\demo It is easy to see check that  $m=(x,t,u,v)\in \Ass(R/J)$, for example because $y=x^2t^2u^3v$ is a non-trivial socle element of $R/J$. Therefore $J^{(n)}=J^n$ for every $n\geq 1$. Then ${\rm depth}(R/J^{(2)})={\rm depth}(R/J^2)=1$ and ${\rm{depth}} (R/J^{(4)})={\rm{depth}} (R/J^{4})=0$, for example by \cite[Example~6.6]{HNTT}. In particular, $m\notin \Ass(R/J^{(2)})$ and $m\in \Ass(R/J^{(4)})=\Ass(R/J^4)$. Hence by Remark ~\ref{Localizations}~(a) below,  $m\notin \Ass(R/(J^{(2)})^{(2)})$ and  therefore $(J^{(2)})^{(2)}\neq J^{(4)}$. \QED

\begin{Remark}\label{Localizations}
Let $J$ be an ideal in a Noetherian ring $S$ and $m\geq 1$ be a positive integer. Then
\begin{enumerate}[$($a$)$]
\item for any $q\in \Ass(S/J^{(m)})$ there exists $p\in \Ass(S/J)$ with $q\subseteq p$;
\item for any $p\in \Ass(S/J)$ one has $\left(J^{(m)}\right)_p=J_p^m$;
\item for any $p\in V(J)$ one has $\left(J^{(m)}\right)_p\subseteq (J_p)^{(m)}$.
\end{enumerate}
\end{Remark}

Despite Example \ref{Counter}, we prove that for any arbitrary ideal $J$ in a Noetherian ring $S$ there exists an integer $m_0=m_0(J)$ such that $(J^{(m)})^{(t)}=J^{(mt)}$ for all $m\geq m_0$ and $t\geq 1$. Of course, when $\Ass(S/J)={\rm Min}(J)$ one can take $m_0=1$. 

\begin{Proposition}\label{m_0}
Let $J$ be an ideal in a Noetherian ring $S$. Then
\begin{enumerate}[$($a$)$]
\item for all $m\geq 1$ and $t\geq 1$ one has $J^{(mt)}\subseteq (J^{(m)})^{(t)}$;
\item there exists $m_0\geq 1$ such that for all $m\geq m_0$ and $t\geq 1$ one has
$$(J^{(m)})^{(t)}=J^{(mt)}.$$
\end{enumerate}
\end{Proposition}

\begin{proof}
(a): Let $x\in J^{(mt)}$. Then by definition there exists $c\in S$ which is a non-zero divisor on $S/J$ such that $cx\in J^{mt}\subseteq (J^{(m)})^t$. By Remark~\ref{Localizations}~(a) we see that $c$ is also a non-zero divisor on $S/J^{(m)}$ and therefore $x\in (J^{(m)})^{(t)}$.\\
(b): Let $\Ass(S/J)=\{p_1,\ldots,p_r\}$; it is well-known 
that there exist integers $m_1,\ldots,m_r$ such that $\Ass(S_{p_i}/J_{p_i}^m)=\Ass(S_{p_i}/J_{p_i}^{m_i})$ for every $m\geq m_i$ \cite{Bro}. Let $m_0=\max\{m_i\}$.
By (a) we only need to prove $\left[ J^{(m)}\right]^{(t)}\subseteq J^{(mt)}$ for all $m\geq m_0$ and $t\geq 1$. It suffices to prove it locally at every associated prime $q$ of $J^{(mt)}$.

By Remark~\ref{Localizations}~(a) there exists $p\in \Ass(S/J)$ such that $q\subseteq p$. By Remark~\ref{Localizations}~(c) and (b) we have
$$\left(\left[ J^{(m)}\right]^{(t)}\right)_p\subseteq \left[\left( J^{(m)}\right)_p \right]^{(t)}=\left[ J_p^m\right]^{(t)}.$$
Now observe that since $q\in \Ass(S/J^{(mt)})$ and $q\subseteq p$, then $q\in \Ass(S_p/(J^{(mt)})_p)$ and then $q\in \Ass(S_p/(J^{mt})_p)$ by Remark~\ref{Localizations}~(b). Since $mt\geq m\geq m_0$, then $q\in \Ass(S_p/J_p^m)$. Therefore, by Remark~\ref{Localizations}~(b) one has $\left[(J_p^m)^{(t)}\right]_q=\left[J_p^{mt}\right]_q=J_q^{mt}=\left[J^{(mt)}\right]_q$.
\end{proof}

We now go back to our original setting.
\begin{Lemma} \label{HomogeneousIdeal1lemma}
Let $R=k[x_0, \ldots, x_N]$, where $k$ is any field. Let  $I$ be a homogeneous ideal in $R$ and assume $\gamma(I)> \frac{\alpha(I)}{N}$.  Then there exists an integer $t_0>0$ such that $I^{(t)}$ satisfies Conjecture~\ref{FMX} for all $t\geq t_0$.
\end{Lemma}

\demo Let $m_0$ be as in Proposition~\ref{m_0}~(b) and write $\gamma(I)=\frac{\alpha(I)}{N}+\epsilon$ for some $\epsilon>0$.
Let  $t_0 \geq \max\{\frac{N-1}{N\epsilon}, m_0\}$. Then if $t\geq t_0$ and $m \geq 1$ we have
\begin{eqnarray*}
\frac{\alpha((I^{(t)})^{(m)})}{m}&=&\frac{\alpha(I^{(tm)})}{m}\geq \gamma(I) \cdot t= \frac{\alpha(I)\,t}{N}+ \varepsilon t\\
&\geq& \frac{\alpha(I)\,t}{N}+  \frac{N-1}{N t_0} t\geq \frac{\alpha(I^{(t)})+N-1}{N}.
\end{eqnarray*}
\QED

We are ready to prove the main result of this section.

\begin{Theorem} \label{HomogeneousIdeal3}
Let $R=k[x_0, \ldots, x_N]$, where $k$ is any field and let $I$ be a nonzero homogeneous ideal in $R$.  Then there exists an integer $t_0>0$ such that $I^{(t)}$ satisfies Conjecture~\ref{FMX} for all $t\geq t_0$.
\end{Theorem}

\demo Let $m_0$ be as in Proposition ~\ref{m_0}~(b) and $m \geq 1$ be an integer. Since $I^m\subseteq I^{(m)}$, then $m \alpha(I)\geq \alpha (I^{(m)})$. First, if for every $s\geq 1$ we have $s \alpha(I)=\alpha (I^{(s)})$,
 then for any $t\geq m_0$ and $m\geq 1$, $$
\frac{\alpha((I^{(t)})^{(m)})}{m}=\frac{\alpha(I^{(tm)})}{m}=\frac{tm \alpha(I)}{m}= t\alpha (I)\geq \alpha (I^{(t)})\geq \frac{\alpha (I^{(t)}) + N-1}{N}.
$$
Next, suppose that there exists $T_1>0$ such that $T_1 \alpha(I)>\alpha (I^{(T_1)})$. Hence, for every $t\geq T_1$ one has $t \alpha(I)>\alpha(I^{(t)})$. Indeed, if $t=T_1+a$ for some $a\geq 0$, then
$$t\alpha(I)=T_1\alpha(I)+a\alpha(I)>\alpha(I^{(T_1)})+a\alpha(I) =\alpha(I^{(T_1)}\cdot I^a)\geq \alpha(I^{(T_1+a)})=\alpha(I^{(t)}),$$ where the last inequality follows from the inclusion $I^{(T_1)}\cdot I^a\subseteq I^{(T_1+a)}$.

Let $t_1=\max\{T_1,m_0\}$ and notice that by the above $t_1\alpha(I)\geq \alpha(I^{(t_1)})+1$. Then for all $m\geq 1$
$$
\frac{\alpha((I^{(t_1)})^{(m)})}{m}=\frac{\alpha(I^{(t_1 m)})}{m}=\frac{\alpha(I^{(t_1 m)})}{t_1 m}\cdot t_1\geq \frac{\alpha(I)\,t_1}{N} \geq \frac{\alpha(I^{(t_1)})+1}{N}.$$
So $\gamma(I^{(t_1)})\geq \frac{\alpha(I^{(t_1)})}{N}+\frac{1}{N} > \frac{\alpha(I^{(t_1)})}{N}.
$
By Lemma~\ref{HomogeneousIdeal1lemma}, there exists $t_2>0$ such that for any  $t\geq t_2$, the ideal $(I^{(t_1)})^{(t)}=I^{(t_1 t)}$ satisfies Conjecture~\ref{FMX}.

Finally, let $t_0$ be an integer such that $t_0 \geq t_1t_2+\frac{\alpha\left(I^{(t_1 t_2)}\right)t_1 t_2}{N-1}$. For any $t\geq t_0$, write $t=(t_1 t_2)q +r$, where $0\leq r< t_1 t_2$; by Lemma~\ref{HomogeneousIdeal2} and the fact that the ideal $I^{(t_1 t_2)}$ satisfies Conjecture~\ref{FMX}, then for all $m\geq 1$ we have
\begin{eqnarray*}
\frac{\alpha((I^{(t)})^{(m)})}{m}&\geq& \frac{\alpha((I^{(t)})^{(t_1 t_2m)})}{t_1 t_2 m}=\frac{\alpha((I^{(t_1t_2)})^{(tm)})}{t m}\cdot \frac{t}{t_1t_2}\\
&\geq& \frac{\alpha(I^{(t_1t_2)})+N-1}{N}\cdot \frac{t}{t_1t_2}=\frac{\alpha(I^{(t_1t_2)})}{t_1t_2}\cdot \frac{t}{N}+\frac{(N-1)t}{Nt_1t_2}\\
&\geq& \left(\frac{\alpha(I^{(t)})}{t}-\frac{\alpha(I^{(r)})}{t}\right)\cdot \frac{t}{N}+\frac{(N-1)t}{Nt_1t_2}=\frac{\alpha(I^{(t)})}{N}-\frac{\alpha(I^{(r)})}{N}+\frac{(N-1)t}{Nt_1t_2}\\
&=&\frac{\alpha(I^{(t)})+N-1}{N}+\frac{(N-1)(t-t_1t_2)-\alpha(I^{(r)}) t_1t_2}{Nt_1t_2}\\
&\geq& \frac{\alpha(I^{(t)})+N-1}{N}+\frac{\alpha(I^{(t_1 t_2)})t_1 t_2-\alpha(I^{(r)}) t_1t_2}{Nt_1t_2}\\
&\geq& \frac{\alpha(I^{(t)})+N-1}{N}.
\end{eqnarray*}
\QED

When $I$ has no embedded components, we have a more explicit description of $t_0$. 
\begin{Corollary}\label{t_0}
If $I$ is a homogeneous ideal with $\Ass(R/I)={\rm Min}(I)$, then one can take $t_0=(N-1)\delta$, where $\delta$ is the first positive integer $s$ with $s\alpha(I)>\alpha(I^{(s)})$.
\end{Corollary}

Although $(N-1)\delta$ is reasonably small, in general it is not the smallest possible $t_0$ for which Theorem \ref{HomogeneousIdeal3} holds. For instance, when $I$ is the ideal of three non collinear points in $\mathbb{P}_{k}^2$, it is easy to see that $\delta=2$. Thus Corollary \ref{t_0} yields that for any $t\geq (N-1)\delta=2$ the ideal $I^{(t)}$ satisfies Conjecture \ref{FMX}; however, it is well-known that $I$ satisfies Conjecture \ref{FMX}. A natural question then arises.

\begin{Question} Let $I$ be a homogeneous ideal in $R$.
Does there exist a number   $t_0=t_0(N)$ such that  $I^{(t)}$ satisfies Conjecture~\ref{FMX} for every $t\geq t_0$?
\end{Question}

Of course, Conjecture~\ref{FMX} is true if and only if the integer $t_0=1$ works for any homogeneous ideal $I$.
Theorem~\ref{main2} says that $t_0=1$ is sufficient for any  finite set of very general points in $\mathbb P_k^N$. The following proposition shows that $t_0=N-1$ is sufficient for any finite set of  points in $\mathbb P_{\mathbb C}^N$.

\begin{Proposition} \label{HomogeneousIdeal11}
Let $I$ be the radical  ideal of a finite set of points in $\mathbb P^N_{\mathbb{C}}$.
 Then $I^{(t)}$ satisfies Conjecture~\ref{FMX} for every $t\geq N-1$.
\end{Proposition}

\demo
By the result of Esnault and Viehweg \cite{EV} one has $\gamma(I)\geq \frac{\alpha(I)+1}{N}= \frac{\alpha(I)}{N}+ \frac{1}{N}$. Set $\varepsilon=\frac{1}{N}$. Then by the proof of Lemma~\ref{HomogeneousIdeal1lemma} (here $m_0=1$ because $I$ is radical) we can take $t_0$ such that $\frac{1}{N}\geq \frac{N-1}{N t_0}$; thus we can take  $t_0= N-1$.
\QED
\bigskip

\subsection{Acknowledgment}
The second and third authors would like to thank the Mathematics Research Communities program, which funded their stay at the University of Kansas in March 2011, where the initial part of this work was developed. All authors would like to thank the MSRI at Berkeley for partial support and an inspiring atmosphere during Fall 2012. Moreover, we would like to thank Craig Huneke and Bernd Ulrich for several helpful conversations.
We are grateful to the anonymous referee, whose careful revision helped us improve the article.

\end{document}